\documentclass[reqno,oneside]{amsart}
\usepackage{amsmath,amssymb,amsthm,graphicx,tikz}
   \usetikzlibrary{arrows.meta}

\usepackage[utf8]{inputenc}
\usepackage[super]{nth}
\usepackage{xcolor}
\usepackage{comment,fancyhdr,tikz,amsmath,amssymb,amsthm,mathtools,centernot,caption,float,graphicx,makecell,array,float,changepage,verbatim,textcomp,parallel,geometry,ragged2e,xcolor,lastpage,bbm,mathrsfs,pgfkeys,mathabx}
\newcommand\join\vee
\newcommand\meet\wedge
    \usepackage[shortlabels]{enumitem}

\usepackage[export]{adjustbox}
\usepackage[
    colorlinks=false,
    citebordercolor=green,
    linkbordercolor=red
    ]{hyperref}

    \usepackage[capitalise]{cleveref}
    \crefformat{equation}{(#2#1#3)}
    \crefformat{enumi}{(#2#1#3)}
    \crefformat{section}{\S#2#1#3}
    \crefformat{subsection}{\S#2#1#3}
    \crefformat{subsubsection}{\S#2#1#3}
    \numberwithin{equation}{section}
    \newcommand{\crefdefpart}[2]{%
        \hyperref[#2]{\namecref{#1}~\labelcref*{#1}~\ref*{#2}}}
    
    \let\etoolboxforlistloop\forlistloop % save the good meaning of \forlistloop
    \usepackage{autonum}
    \let\forlistloop\etoolboxforlistloop % restore the good meaning of \forlistloop
    \makeatletter
    \newcommand{\blx@noerroretextools}{}
    \makeatother
    
    \usepackage[
        backend=biber,
        style=alphabetic,
        maxbibnames=50,
        maxalphanames=50,
        maxcitenames=50
        ]{biblatex}
    \newcommand{\Z}{\mathbb{Z}}
    
    \newcommand{\R}{\mathbb{R}}

    \def\sph^#1{\mathbb S^{#1}}
    \newcommand\p[1]{\left(#1\right)}
    
    \renewcommand\b[1]{\left[#1\right]}

    \newcommand\jp[1]{\left\langle#1\right\rangle}
    
    \newcommand\eps\varepsilon
    \newcommand\pji\varphi
    \newcommand\sig\varsigma

    \newcommand\beq{\begin{equation}}
    \newcommand\eeq{\end{equation}}

    \def\XXint#1#2#3{{\setbox0=\hbox{$#1{#2#3}{\int}$ }
    \vcenter{\hbox{$#2#3$ }}\kern-.6\wd0}}

    \newtheoremstyle{thmst}% Theorem
        {5pt}% Space above
        {5pt}% Space below
        {}% Body font 
        {}% Indent amount
        {\bfseries}% ⟨Theorem head font⟩
        {.}% ⟨Punctuation after theorem head ⟩
        {.5em}% Space after theorem head 
        {}%
    \newtheoremstyle{rmkst}% Remark
        {5pt}% Space above
        {5pt}% Space below
        {}% Body font 
        {}% Indent amount
        {\bfseries}% ⟨Theorem head font⟩
        {.}% ⟨Punctuation after theorem head ⟩
        {.5em}% Space after theorem head 
        {}%
    \newtheoremstyle{lmst}% Lemma
        {5pt}% Space above
        {5pt}% Space below
        {}% Body font 
        {}% Indent amount
        {\bfseries}% ⟨Theorem head font⟩
        {.}% ⟨Punctuation after theorem head ⟩
        {5pt plus 1pt minus 1pt}% Space after theorem head 
        {}%
    \newtheoremstyle{prpst}% Proposition
        {5pt}% Space above
        {5pt}% Space below
        {}% Body font 
        {}% Indent amount
        {\bfseries}% ⟨Theorem head font⟩
        {.}% ⟨Punctuation after theorem head ⟩
        {.5em}% Space after theorem head 
        {}%
    \newtheoremstyle{defst}% Proposition
        {5pt}% Space above
        {5pt}% Space below
        {}% Body font 
        {}% Indent amount
        {\bfseries}% ⟨Theorem head font⟩
        {.}% ⟨Punctuation after theorem head ⟩
        {.5em}% Space after theorem head 
        {}%
    \newtheoremstyle{wlst}% White Lie
        {5pt}% Space above
        {5pt}% Space below
        {}% Body font 
        {}% Indent amount
        {\bfseries}% ⟨Theorem head font⟩
        {.}% ⟨Punctuation after theorem head ⟩
        {.5em}% Space after theorem head 
        {}%
    \newtheoremstyle{asmpst}% White Lie
        {5pt}% Space above
        {5pt}% Space below
        {}% Body font 
        {}% Indent amount
        {\bfseries}% ⟨Theorem head font⟩
        {.}% ⟨Punctuation after theorem head ⟩
        {.5em}% Space after theorem head 
        {}%
        
    \theoremstyle{thmst}
        \newtheorem{theorem}{Theorem}[section]
        \newtheorem{conjecture}[theorem]{Conjecture}
        
    \theoremstyle{rmkst}
        \newtheorem{remark}[theorem]{Remark}
    \theoremstyle{defst}
        
    \theoremstyle{defst}
        
    \theoremstyle{lmst}
        \newtheorem{lemma}[theorem]{Lemma}
    \theoremstyle{prpst}
        
    \theoremstyle{wlst}
        
    \theoremstyle{asmpst}
        
    \theoremstyle{thmst}

        \newlist{defenum}{enumerate}{1} % should only occur inside definition env.
        \setlist[defenum]{label=(\roman*),ref=\thedefinition\,(\roman*)}
        \crefname{defenumi}{Definition}{Definitions}
        
        \newlist{lemenum}{enumerate}{1} % should only occur inside definition env.
        \setlist[lemenum]{label=(\roman*),ref=\thelemma\,(\roman*)}
        \crefname{lemenumi}{Lemma}{Lemmas}
        
        \newlist{corenum}{enumerate}{1} % should only occur inside definition env.
        \setlist[corenum]{label=(\roman*),ref=\thecorollary\,(\roman*)}
        \crefname{corenumi}{Corollary}{Corollaries}
        
        \newlist{thmenum}{enumerate}{1} % should only occur inside definition env.
        \setlist[thmenum]{label=(\roman*),ref=\thetheorem\,(\roman*)}
        \crefname{thmenumi}{Theorem}{Theorems}
        
        \newlist{rmkenum}{enumerate}{1} % should only occur inside definition env.
        \setlist[rmkenum]{label=(\roman*),ref=\theremark\,(\roman*)}
        \crefname{rmkenumi}{Remark}{Remarks}
        
        \newlist{prpenum}{enumerate}{1} % should only occur inside definition env.
        \setlist[prpenum]{label=(\roman*),ref=\theproposition\,(\roman*)}
        \crefname{prpenumi}{Proposition}{Propositions}
        
        \newlist{axenum}{enumerate}{1} % should only occur inside definition env.
        \setlist[axenum]{label=(\roman*),ref=\theaxiom\,(\roman*)}
        \crefname{axenumi}{Axiom}{Axioms}
        
        \newlist{defcrit}{enumerate}{1} % should only occur inside definition env.
        \setlist[defcrit]{label=(\roman*),ref=\thedefinition\,(\roman*)}
        \crefname{defcriti}{Definition}{Definitions}
        
        \newlist{lemcrit}{enumerate}{1} % should only occur inside definition env.
        \setlist[lemcrit]{label=(\alph*),ref=\thelemma\,(\alph*)}
        \crefname{lemcriti}{Lemma}{Lemmas}
        
        \newlist{thmcrit}{enumerate}{1} % should only occur inside definition env.
        \setlist[thmcrit]{label=(\alph*),ref=\thetheorem\,(\alph*)}
        \crefname{thmcriti}{theorem}{theorems}
        
        \newlist{rmkcrit}{enumerate}{1} % should only occur inside definition env.
        \setlist[rmkcrit]{label=(\alph*),ref=\theremark\,(\alph*)}
        \crefname{rmkcriti}{remark}{remarks}
        
        \newlist{prpcrit}{enumerate}{1} % should only occur inside definition env.
        \setlist[prpcrit]{label=(\alph*),ref=\theproposition\,(\alph*)}
        \crefname{prpcriti}{proposition}{propositions}

        \makeatletter
        \let\ifnc\@ifnextchar
        \makeatother

            \reversemarginpar
            \def\redit {\marginpar{\raggedleft\color{red}{See edit $\implies$}}\color{red}}
            \def\rs#1.{\redit #1.\color{black}}
            \def\rsm#1.{{\color{red} #1.}}
            
            \def\bedit {\marginpar{\raggedleft\color{blue}{See edit $\implies$}}\color{blue}}
            \def\bs#1.{\bedit #1.\color{black}}
            \def\bsm#1.{{\color{blue} #1.}}

            \def\ind#1_#2{\left\{#1_{#2}\right\}}
            \def\<#1>{\jp{#1}}
            \def\-#1/{{}_{#1}}
            \makeatletter
                \newcommand\pl@write[3]{%
                    \left\|#1\right\|_{L^{#2}#3}%
                    }
                \def\pl #1__#2{%
                    \def\pl@arg@i{#1}%
                    \def\pl@arg@ii{#2}%
                    \def\pl@arg@iii{\alpha}%
                    \futurelet\next\pl@eval%
                    }
                \def\pl@eval{%
                    \ifx\next\bgroup%
                            \expandafter\pl@eval@iii%
                        \else%
                            \expandafter\pl@eval@ii%
                        \fi%
                    }
                \def\pl@eval@iii#1{%
                    \pl@write\pl@arg@i\pl@arg@ii{\p{#1}}%
                    }
                \def\pl@eval@ii{%
                    \pl@write\pl@arg@i\pl@arg@ii{}%
                    }
            \makeatother
            \makeatletter
                \newcommand\ef@write[3]{%
                    ^{#1\frac{#2}{#3}}{}%
                    }
                \newcommand\Ef@write[3]{%
                    ^{#1#2/#3}{}%
                    }
                \def\ef #1{%
                    \def\ef@arg@i{}%
                    \def\ef@arg@ii{#1}%
                    \futurelet\next\ef@eval%
                    }
                \def\ief #1{%
                    \def\ef@arg@i{-}%
                    \def\ef@arg@ii{#1}%
                    \futurelet\next\ef@eval%
                    }
                \def\ef@eval{%
                    \ifx\next/%
                            \expandafter\ef@eval@iii%
                        \else%
                            \expandafter\ef@eval@ii%
                        \fi%
                    }
                \def\ef@eval@iii/{%
                    \expandafter\ef@eval@v%
                    }
                \def\ef@eval@v#1{%
                    \Ef@write\ef@arg@i\ef@arg@ii{#1}}
                \def\ef@eval@ii{%
                    \expandafter\ef@eval@iv%
                    }
                \def\ef@eval@iv#1{%
                    \ef@write\ef@arg@i\ef@arg@ii{#1}%
                    }
            \def\e@writep #1{%
                ^{+#1}{}%
                }
            \def\e@writen #1{%
                ^{-#1}{}%
                }
            \def\e #1{%
                \ifx#1-%
                        \expandafter\e@writen%
                    \else%
                        \ifx#1+%
                                \expandafter\e@writep%
                            \else%
                                ^{#1}{}%
                            \fi%
                    \fi%
                }
            \def\e@2 {%
                }

            \makeatother
            \makeatletter
            \newcommand{\undersim}[1]{\mathrel{\mathpalette\@undersim{#1}}}
            \newcommand{\@undersim}[2]{%
              \vcenter{%
                \ialign{%
                  ##\cr
                  $\m@th#1#2$\cr
                  \noalign{\nointerlineskip\kern.2ex}
                  $\m@th#1\sim$\cr
                  \noalign{\kern-.4ex}
                }%
              }%
            }
            \makeatother

\definecolor{packetblue}{RGB}{33,95,166}
\definecolor{packetred}{RGB}{200,78,53}

\title{On $C^k$ convex variants of Stein's Weighted $L^2$ Conjecture for Bochner-Riesz means}
\author{Ruixiang Zhang}
 \thanks{University of California, Berkeley and Institute for Advanced Study, Email: \texttt{ruixiang@berkeley.edu, rzhang@ias.edu}}
\date{\today}

\begin{document}
\begin{abstract}
     We prove power-loss results for some variants of Stein's Weighted $L^2$ Conjecture for Bochner-Riesz means. The variants concern multipliers by convex bodies with $C^k$ boundaries in every dimension $n$. The counterexamples are related to  recently-constructed counterexamples to the Mizohata-Takeuchi Conjecture. Related results like logarithmic blowups for the endpoint case of the original conjecture are also discussed.
\end{abstract}
\maketitle

\section{Introduction}

In 1979,  Stein asked  if  one can expect to use Kakeya or Nikodym type maximal functions to control the ball multiplier or Bochner–Riesz multiplier operators by weighted $L^2$  inequalities (Conjecture \ref{Steinconj} below, see Problem 5 (b) and (c) in \cite{stein-conjecture-79}). Stein's question was motivated by the importance of the ball multiplier and Bochner-Riesz means. It is known that these operators fail to satisfy certain $L^p$ bounds that may have appeared reasonable (for example, Fefferman \cite{fefferman1971multiplier} disproved $L^p$ boundedness of the ball multiplier for all $p \neq 2$ and all dimensions $>1$ using the existence of measure-zero Kakeya sets). Because of the exotic $L^p$ behaviors, Stein asked whether one can hope for a substitute boundedness in terms of weighted $L^2$ spaces: can one find  natural weighted $L^2$ estimates for the ball multiplier and the Bochner-Riesz means, perhaps via the Kakeya/Nikodym type maximal functions?

In this paper, we will show that due to existence of the counterexample to the Mizohata-Takeuchi Conjecture \cite{cairo2025power} (along with its proof), some cautions have to be taken if one wants to replace the ball multiplier by a strictly convex body multiplier with $C^k$ boundary. This is related to the existence of  many points in a higher-rank generalized arithmetic progression on the boundary of the convex body, and demonstrates the role of these objects in the theory of Bochner-Riesz-type operators.

For technical simplicity, we will focus on Stein's Problem 5(c), $\varepsilon=\delta=0$ case below, but remark that one can have a similar discussion for the other parts of problems in 5(b) and (c). Some of this is recorded in Section \ref{variants} below.

\begin{conjecture}[Problem 5(c) in \cite{stein-conjecture-79}, $\varepsilon=\delta=0$ case]\label{Steinconj}
    Let $S$ be the disc multiplier in $\R^2$: \[(Sf)^{\hat{}} = 1_{B_1} \hat{f}.\]

    Is it true that for every $f\in L_{loc}^1 (\R^2)$ and every continuous weight function $w$, \[\int_{\R^2} |Sf (x)|^2 w(x) \mathrm{d}x \lesssim \int_{\R^2} |f(x)|^2 \mathcal{M} w (x) \mathrm{d}x,\] where $\mathcal{M}$ is the Nikodym-type maximal function: \[\label{defnofmaximalfunction}\mathcal{M}w(x) = \sup_{r>0, \theta}\frac{1}{2r}\int_{-r}^r w(x-t(\cos \theta, \sin \theta))\mathrm{d}t.\]
\end{conjecture}

We will show that this conjecture and all higher-dimensional analogues are false when the unit disc is replaced by some suitable convex body with $C^k$-boundary, for any prescribed $k$.

\begin{theorem}\label{Counterexthm}
    For any given dimension $n \geq 2$ and smoothness $k\geq 2$, there exists a compact convex body $K \subset \R^n$ with $C^k$-boundary, such that if we define the multiplier by $1_K$ as \[\label{defnofSK}(S_K f)^{\hat{}}=1_K \cdot \hat{f},\] the inequality \[\label{failedestimateforSK}\int_{\R^n} |S_Kf (x)|^2 w(x) \mathrm{d}x \lesssim \int_{\R^n} |f(x)|^2 \mathcal{M} w (x) \mathrm{d}x\] cannot always hold. More quantitatively, for any fixed $\varepsilon>0$ we can find such a $K$ such that for every sufficiently large $R\gtrsim_{n, k, \varepsilon} 1$, there is an $f$ and $w$ supported in $B_R$ such that 
    \[\label{correctededestimateforSK}\int_{\R^n} |S_Kf (x)|^2 w(x) \mathrm{d}x \geq R^{\frac{n-1}{n-1+k}-\eps} \int_{\R^n} |f(x)|^2 \mathcal{M} w (x) \mathrm{d}x.\]
\end{theorem}

\subsection{Ideas behind the counterexample}

This counterexample $(f, w)$ for Conjecture \ref{Steinconj} was obtained by a construction by Cairo and the author in the recent power-blowup refutation to the Mizohata-Takeuchi Conjecture for $C^k$ hypersurfaces \cite{cairo2025power}, plus a ``shifting trick'' that is in spirit similar to what Fefferman used in his proof of the unboundedness of the ball multiplier in \cite{fefferman1971multiplier}. In addition to the  two main ideas above, some technicality is needed in our current proof to estimate an $L^2$ mass efficiently. We will give a detailed outline of the proof in Section \ref{outlinesec}.

\subsection{Historical background}

Stein's original motivation in proposing Conjecture \ref{Steinconj} was to find suitable maximal functions to have correct weighted $L^2$ estimates for the Bochner-Riesz operator. One natural candidate was the Nikodym maximal function, especially because its ubiquity in higher dimensional Fourier analysis and  (conjectural) numerically coincidental $L^p$ boundedness properties. However, because of this counterexample, one may want to include maximal functions based on generalized arithmetic progression (GAP)s for an alternative formulation of Conjecture \ref{Steinconj}, especially if they wish the same inequality to hold for all convex body multipliers.\footnote{This direction also came up in discussions with Larry Guth and Hong Wang, and with Hannah Cairo. The author would like to thank them for helpful discussions.} We caution that the recent refutation of the Unit Distances Conjecture \cite{openai-unit-distances-26} (along with related works like \cite{bssz-sum-product-26}) newly unveiled some exotic behaviors of higher rank GAPs' role in analysis, and it may even be possible that the most reasonable conjecture can involve objects other than bounded rank GAPs.

As a comparison, we have the following theorem for convex polytope multipliers in every dimension:

\begin{theorem}\label{SingleNikodym}
For every given convex compact polytope $K \subset \R^n$, define $S_K$ by \ref{defnofSK}, then
    \[\label{correcestimateforSK}\int_{\R^n} |S_Kf (x)|^2 w(x) \mathrm{d}x \lesssim_{K, s} \int_{\R^n} |f(x)|^2 (\mathcal{M} w^s (x))^{\frac{1}{s}} \mathrm{d}x, \forall s>1.\] %where $M$ is the strong maximal function and $s>1$ is arbitrary.
\end{theorem}

\begin{proof}
A classical result of C\'{o}rdoba (The author learned it from \cite{cordoba1981spherical}, Theorem 4, where the authors attributed the result to \cite{cordoba1981some}) implies \eqref{correcestimateforSK} with $K$ being a unit box\footnote{In C\'{o}rdoba's original result, it suffices to use the strong maximal function, where only coordinate-axes-parallel boxes were considered. The strong maximal function is known to be dominated by $\mathcal{M}$ up to a dimensional constant.  We also note that C\'{o}rdoba's original theorem was a stronger square function estimate.}, hence also for $K$ being any parallelepiped by affine transformation. By a limiting argument, \eqref{correcestimateforSK} also holds for $K$ being any orthant, cylinders over octants in lower dimensions and any of their affine transformations. The result can then be seen by writing a general $1_{K}$ as a linear combination of characteristic functions of tangent cones by the Brianchon-Gram identity (see for example equation (42) in \cite{gravin2012inverse}, Appendix A), which can further be triangulated into affine transformations of orthants or cylinders over orthants in lower dimensions.
\end{proof}

Theorem \ref{SingleNikodym} should be compared with Theorem \ref{Counterexthm} to illustrate different behaviors of polytope multipliers and general convex sets multipliers. A similar proof shows that one can in fact take $K$ to be any polytope (not necessarily convex or compact) in Theorem \ref{SingleNikodym}. The author was not able to find Theorem \ref{SingleNikodym} in its current form in the literature and presents a self-contained proof here, but notes that the proof of it is closely related to classical work of C\'{o}rdoba \cite{cordoba1982geometric}, C\'{o}rdoba-Fefferman \cite{cordoba1976weighted} (whose results also have endpoint counterparts by Wilson \cite{wilson1989weighted} and P\'{e}rez \cite{perez1994weighted}). For related context in singular integration theory, see e.g. the more recent works \cite{bennett2014optimal, beltran2018fefferman} and references therein.

For related history in weighted norm inequalities, Rubio de Francia developed general principles
connecting weighted $L^2$ inequalities with $L^p$ boundedness
(see for example \cite{rubio-weighted-vector-82} and his
boundedness principle). In $\mathbb{R}^2$, Carbery
\cite{carbery-weighted-maximal-br-85} proved weighted $L^2$
inequalities for Bochner--Riesz means of every positive order,
constructing a family of substitutes for $\mathcal{M}$,
including one bounded on $L^r$ for $2\leq r\leq4$.
Later, Carbery and Seeger \cite{carbery-seeger-weighted-br-00}
constructed, for each positive order, a single substitute
bounded on $L^r$ throughout $2\leq r\leq\infty$
(and slightly below $2$, depending on the order).
Their improvement combined Carbery's method with ideas from
Seeger's work \cite{seeger-endpoint-br-96} on endpoint
Bochner--Riesz.
These controlling weight operators differ from the Nikodym
maximal operator.

\subsection{Organization of the paper}

This paper is organized as follows: We first give an outline of the proof of Theorem \ref{Counterexthm} in Section \ref{outlinesec}. Section \ref{proofsec} has the full detailed constructional proof for Theorem \ref{Counterexthm}. Section \ref{variants} discusses consequences of our proof technique on some variants of Conjecture \ref{Steinconj}.

\section*{Acknowledgments} This work is supported by NSF CAREER DMS–2143989 and Sloan Research Fellowship. The author would like to thank Anthony Carbery and Andreas Seeger for helpful discussions, especially  on the historical context and some benchmarks. He would also like to thank Hannah Cairo for helpful discussions. GPT-5.6 Sol was used when the author tried to search the literature for a theorem like Theorem \ref{SingleNikodym}. The current proof of it was found by GPT-5.6 Sol and rewritten by the author. GPT-5.6 Sol was also used in preparing the illustrative figure at the end of Section \ref{proofsec}. Appendix A gives some caution on the endpoint case of Stein's original conjecture via an elementary construction, and is generated by GPT-6 Astra in the proofreading stage of the present article.

\section{Outline of the proof of Theorem \ref{Counterexthm}}\label{outlinesec}

In this section, we outline the proof of Theorem \ref{Counterexthm}. It is known that Stein's conjecture is related to the Mizohata-Takeuchi (MT) Conjecture, which was recently disproved \cite{cairo-counterexample-25, cairo2025power}. As pointed out by Cairo \cite{cairo-counterexample-25}, disproof of the MT conjecture leads to the disproof of the following Stein-like conjecture, stated and studied by e.g. \cite{crs-radial-mt-92, carbery-tubes-09, bcsv-stein-conjecture-06, bs-stein-mt-xray-21, ciw-mt-24, bennett2024tomographic, bgno-phase-space-24, cairo-counterexample-25}:

\begin{conjecture}[Disproved in \cite{cairo-counterexample-25}, \cite{cairo2025power} further confirmed possible power loss for some $\Sigma$]\label{Steinlikeconj}
    Let $\Sigma$ be a compact $C^2$ strictly convex hypersurface patch in $\R^n$ with  volume measure $\mathrm{d}\sigma$ and normal vector $n(\xi)$ at $\xi \in \Sigma$. Let $w \geq 0$ be a weight. Then \[\int \left|\int_{\Sigma} f(\xi)e^{\mathrm{i}x\cdot \xi}\mathrm{d}\sigma (\xi)\right|^2w(x)\mathrm{d}x\lesssim_{\Sigma} \int_{\Sigma} \left(\sup_{T: 1-\text{tube parallel to } n(\xi)} \int_T w\right) \cdot |f(\xi)|^2 \mathrm{d}\sigma(\xi).\]
\end{conjecture}

In many harmonic analysis problems, multipliers for convex bodies are related to Fourier extension operators ($f \mapsto \int_{\Sigma} f(\xi)e^{\mathrm{i}x\cdot \xi}\mathrm{d}\sigma (\xi)$). One connection can be seen by taking the difference of two slightly rescaled convex body multipliers, resulting in a multiplier by a neighborhood of a hypersurface. Conjecture \ref{Steinlikeconj} is related to Conjecture \ref{Steinconj} in this way and the two are often studied alongside each other. Nevertheless,   the disproof of Conjecture \ref{Steinlikeconj} does not automatically disprove Conjecture \ref{Steinconj}: One can view Conjecture \ref{Steinlikeconj} as the case where $r$ is ``maximal'' in the definition of $\mathcal{M}$. But when $\mathcal{M}$ has a supremum taken over all scales, more is required for a disproof. We overcome this difficulty by an adaptation of a ``shifting trick'' used in Fefferman's disproof of the boundedness of the Ball Multiplier in $L^p$ ($p\neq 2$) \cite{fefferman1971multiplier}.

\subsection{A construction from \cite{cairo2025power}}

In order to prove Theorem \ref{Counterexthm}, we need a lemma in \cite{cairo2025power} where a power-loss counterexample to the closely-related Mizohata-Takeuchi Conjecture was constructed.

\begin{lemma}\label{CZcountinglemma}
    Let $\eps>0$ be an arbitrarily small fixed constant and the dimension $n\geq 2$ and smoothness order $k\geq 2$ fixed. Then there exists a large dimension $N$, a compact convex body $K\subset \R^n$ that contains $0$ in its interior, with $C^k$ strictly convex boundary $\partial K = \Sigma$,  and a sequence of positive integers $R_j \to \infty$, such that:

    For every $R_j$, there exists a projection $\pi_j: \R^N \to \R^n$, an affine transform $\tau_j: \R^n \to\R^n$ that is a $r_j$-times rescaling of a rigid motion, $r_j \in [1, 2]$, such that \[\left|\mathcal{N}_{\frac{1}{R_j}}(\Sigma) \cap (\tau_j\pi_j (\frac{1}{R_j^{\frac{1}{N}}}\Z^N))\right|\geq R_j^{\frac{n-1}{n-1+k}-\eps},\] where $\mathcal{N}_{\frac{1}{R_j}}(\Sigma)$ denotes the $\frac{1}{R_j}$-neighborhood of $\Sigma$, and that every hyperplane in $\R^n$ is $\frac{1}{R_j}$-close to $O_{\eps}(R_j^{\eps})$ many points in $\pi_j (\frac{1}{R_j^{\frac{1}{N}}}\Z^N)$. 

Moreover we have the \emph{separation condition}: any two points in $\pi_j (\frac{1}{R_j^{\frac{1}{N}}}\Z^N \cap B_{100})$ are $R_j^{-\frac{1.5}{n}}$-separated.\footnote{The $\frac{1.5}{n}$ here is not sharp. From the proof we see anything $>\frac{1}{n}$ suffices.}
\end{lemma}

\begin{proof}
    All properties except the separation condition are mostly done in the proof of \cite[Proposition 5.3 and Lemma 5.4]{cairo2025power}. We only explain some minor differences/additional steps. First of all, note that one can focus on the construction of $\Sigma$. As long as we construct a compact, strictly convex $C^k$ $\Sigma$ satisfying the counting needed, we can construct $K$ by using a $C^k$ extension of $\Sigma$ as its boundary. One can construct $\Sigma$ as in \cite[Proposition 5.3]{cairo2025power}, with the desired many points at a single scale $R$. This can be then upgraded to a sequence of scales by a standard sequential perturbation argument in the proof of \cite[Lemma 5.4]{cairo2025power}.

    To obtain the separation condition, it suffices to notice that by \cite[Proof of Proposition 5.3]{cairo2025power}, there is a fixed choice of $N$ such that for each sufficiently large $R_j$ chosen there are $>99\%$ of the $\pi_j$ that work. Now for a random $\pi_j$, the probability for a unit vector's projection under $\pi_j$ less than $\eps$ is $O_N(\eps^{n})$. A union bound over the lattice differences gives the probability of $\pi_j$ violating the separation condition is $O_N(R_j^{-0.5}+R_j^{-1.5+\frac{n}{N}})=o_N(1)$.
\end{proof}

\subsection{A detailed outline of the proof of Theorem \ref{Counterexthm}.} Our proof of Theorem \ref{Counterexthm} uses the $K$ constructed in Lemma \ref{CZcountinglemma} where many rescaled projected lattices in $\tau_j\pi_j (\frac{1}{R_j^{\frac{1}{N}}}\Z^N))$ lie on its boundary. For fixed $\eps>0$, apply that lemma to find $N$, $\{R_j\}$ and $K$ and fix the $N$ henceforth in the whole proof. Note the set $\tau_j \pi_j (\frac{1}{R_j^{\frac{1}{N}}}\Z^N))$, when translated to contain $0$, is a subgroup of $\R^n$, and we will crucially use this property.

To motivate the construction, take a bounded patch of $\tau_j\pi_j (\frac{1}{R_j^{\frac{1}{N}}}\Z^N))$ and call it $L$. We also consider the translation $L_0$ of $L$ centered at the origin that becomes an approximate subgroup. We will use a function $\psi_1$ that is a mollified (at scale $\frac{1}{R_j}$) and truncated version of $1_{L_0}$ and try to use $|\hat{\psi_1}|^2$ as the weight $w$ and $\hat{\psi_1}$ as the test function $f$. For convenience, we will use another function $\psi_2$ that is a mollified and truncated version of $1_{\mathcal{N}_{\frac{1}{R_j}}(\Sigma) \cap L}$ and a function $F = \hat{\psi_2}$.

Let us pause and digest and visualize $\hat{\psi}_1$. One way is to use the projection-slicing theorem: $\hat{\psi}_1$ can be roughly visualized as an $n$-dimensional slicing of an $N$-dimensional lattice-based smoothed unit-ball collection inside an $R_j$-ball.

Our constructions were inspired by the study of Mizohata-Takeuchi Conjecture. We now explain the motivation of our choice of $f$ and $w$, and the connection. If the multiplier $S_K$ satisfies the bound \eqref{failedestimateforSK}, then by rescaling, two similar multipliers $S_{(1-\frac{1}{R_j})K}$ and $S_{(1+\frac{1}{R_j})K}$ also satisfy the same bound, and so does their difference, which is a multiplier by $1_{\mathcal{N}_{\frac{1}{R_j}}(\Sigma)}$ that we denote $S_{\mathcal{N}_{\frac{1}{R_j}}(\Sigma)}$. The situation is now related to  the MT Conjecture as motivated before. We would like to disprove

\[\label{MTversion}\int_{\R^n} |S_{\mathcal{N}_{\frac{1}{R_j}}(\Sigma)}f (x)|^2 w(x) \mathrm{d}x \lesssim \int_{\R^n} |f(x)|^2 \mathcal{M} w (x) \mathrm{d}x.\]

Now if we use $f=\hat{\psi_1}$ and $w=|\hat{\psi_1}|^2$, then both sides of \eqref{MTversion} can be explicitly evaluated using Plancherel on the Fourier side. Note that $S_{\mathcal{N}_{\frac{1}{R_j}}(\Sigma)}f$ is roughly the function $F$. We see the left-hand-side integral in \eqref{MTversion} will become an integral of $|\psi_1 * \psi_2|^2$. This has a lot of built-in collisions due to the approximate group structure of $L_0$, and will consequently become much (a power of $R_j$ times) larger than the integral $\int_{\R^n} |f(x)|^2 \frac{w(x)}{R_j}\mathrm{d}x$ in the same way as the blowup calculation for MT in \cite{cairo2025power}. This leads us to see a suspicious aspect of \eqref{MTversion}: if we only take the averages over lines of length $\sim R_j$ in the definition of $\mathcal{M}$, indeed we can hope this modification of $\mathcal{M}(x)$ to be bounded by something like $\frac{w(x)}{R_j^{1-\eps}}$ because of the non-concentration of $L$ around hyperplanes provided by Lemma \ref{CZcountinglemma}. Then \eqref{MTversion} would ``fail'' in the same way as power-losses for MT in \cite{cairo-counterexample-25}. However, recall that in the definition of $\mathcal{M}$ we must take line segments of all lengths including very short ones. This prevents MT-type considerations from entering the picture. Indeed, it is unrealistic to hope $\mathcal{M}w$ to be bounded by $\frac{w}{R_j^{1-\eps}}$ everywhere for any $w$ because almost every point is a Lebesgue point in every direction and we almost surely have $\mathcal{M}w \geq w$.

It turns out that there is a way to bypass the above difficulty and construct a genuine counterexample, and the fix resembles Fefferman's classical disproof \cite{fefferman1971multiplier} of the Ball Multiplier Conjecture. The key is to take $\text{supp}f$ and the $\text{supp}w$ to be $\sim R_j$-separated. This can be ensured by e.g. truncating $\hat{\psi_1}$ inside $B_{R_j}$ and use it as $f$. Recall this choice has a mollification in $\hat{f}$. If we replace that mollification to be at a scale $Q$-times larger than $\frac{1}{R_j}$ ($Q$ a large constant to be determined), corresponding to shrinking $\text{supp} f$ by $\frac{1}{Q}$, then $S_{\mathcal{N}_{\frac{1}{R_j}}(\Sigma)} f$ will be different from $f$. By heuristic from the uncertainty principle, $S_{\mathcal{N}_{\frac{1}{R_j}}(\Sigma)} f$ should spread out and we can then take $w$ to be the truncation of $|\hat{\psi_1}|^2$ outside of $B_{\frac{2}{Q}R_j}$ but inside $B_{R_j}$. Temporarily suppose we can make this $\int_{\R^n} |S_{\mathcal{N}_{\frac{1}{R_j}}(\Sigma)} f(x)|^2 w(x) \mathrm{d}x$ as large as before. Notice that since the support of $f$ is quantitatively separated from the support of $w$, we see when calculating $\mathcal{M}w$, one has to take the segment length $\gtrsim R_j$ inside $\text{supp} f$. This will make the maximal function in $\text{supp} f$ much smaller, matching something like $O_{Q} \left(\frac{\|w\|_{L^{\infty}}}{R_j}\right)$ in the last paragraph, and finally lead to a power blowup of \eqref{MTversion}.

The only loose end is how to lower bound $\int_{\R^n} |S_{\mathcal{N}_{\frac{1}{R_j}}(\Sigma)} f(x)|^2 w(x) \mathrm{d}x$. Even though we have the uncertainty principle intuition, this quantity does not seem explicitly computable and here is the technical method to bound it: We compute it on the Fourier side, and by Plancherel we will see that this amounts to showing some mass around $\mathcal{N}_{\frac{1}{R_j}}(\Sigma)\cap \mathcal{N}_{\frac{Q}{R_j}} (L)$ has larger $L^2$ norm when convolved with a function like $1_{\mathcal{N}_{\frac{1}{R_j}} (L)}$, compared with convolution with a function like $Q^{-n}\cdot 1_{\mathcal{N}_{\frac{Q}{R_j}} (L)}$. For technical reasons, we did not find this so easy to prove either. The good news is that if we shrink the Fourier support of $f$ to only in a $\frac{1}{Q}$-box, then all the points in $L\cap \mathcal{N}_{\frac{1}{R_j}} (\Sigma)$ inside this box will have nearby tangent direction of $\Sigma$ lining up, and this can help us prove the above-mentioned favorable comparison and finish the proof.

\section{Proof of Theorem \ref{Counterexthm}}\label{proofsec}

\begin{proof}[Proof of Theorem \ref{Counterexthm}]
    In this proof, we will first reduce from $S_K$ to a multiplier associated with a thin neighborhood of $\partial K$ and use the construction in Lemma \ref{CZcountinglemma} (where $\Sigma$ is part of $\partial K$). We  remark that the affine transformation $\tau_j$ in Lemma \ref{CZcountinglemma} is mild and harmless.  For ease of notation, when discussing the multiplier associated with $\partial K \supset \Sigma$, we always assume $\tau_j = \text{id}$. It is clear how to modify our construction to accommodate the effect of $\tau_j$ in reality.
    
    We first fix an $\eps>0$  and choose $N, K, \Sigma, \{R_j\}$ depending on  the $n, k, \eps$ as in Lemma \ref{CZcountinglemma}. Fix these parameters throughout (so in this proof in $\lesssim$, etc. we will suppress the dependence on $n, N, K$, etc.). In the proof, we will work both in $\R^N$ and in $\R^n$ and will work on the physical and the Fourier spaces in these dimensions. We will use $B_r (x)$ to denote the $r$-ball around $x$ in these spaces and simply use $B_r$ for $B_r (0)$.

    We will use several truncations, dually mollifications on the Fourier side. We fix a compactly supported smooth function $\phi\geq 0$ with $1\leq \phi\leq C_N$ on $B_1$, $\text{supp}\phi \subset B_{2}$ and $\hat{\phi}\geq 0$ on $\R^N$ and a smooth function $0\leq \psi\leq 1$ with $\psi=1$ on $B_1$, $\text{supp}\psi \subset B_{2}$ on $\R^n$. We caution that we will use $\phi$ to make cutoffs on both the spatial side and the Fourier side, and choose to only work with one function $\phi$ to avoid notational complications. We will also use the following notation: For a tempered distribution $F$ (that will be on the Fourier side in applications) on $\R^N$ and $R, S>0$, we define \[F_{R; S} = \left(F\phi(\frac{\cdot}{S})\right)*(R^N \hat{\phi} (R\cdot)).\] Intuitively, this is ``localizing on the Fourier side at scale $S$ and then localizing on the physical side at scale $R$''. It results in a compactly supported function in the physical $\R^N$, i.e. $\check{F}_{R, S}$ is compactly supported. We have the following more precise formula of $\check{F}_{R, S}$: It is given by a cutoff after a convolution:
    \[(F_{R; S})^{\check{}} = \left(\check{F}*(S^N\check{\phi}(S\cdot))\right)\cdot \phi (\frac{\cdot}{R}).\]

    For a discrete set $X$, by a slight abuse of notation we also use $X$ to denote the tempered distribution $\sum_{x\in X} \delta_x$. For a tempered distribution $b$, we denote the translation of $b$  by a vector $x$ to be $T_x b$.

    We will construct counterexamples to \eqref{failedestimateforSK}. By taking the difference between $S_K$ and another similar multiplier by a small rescaling of $1_K$, we see it suffices to construct explicit  $f=f_j$ and $w=w_j$  to fail \eqref{MTversion}. We will use a large constant $Q>100$ to be determined later. Before fixing $Q$, we will always track and specify the dependence of everything on $Q$. For each sufficiently large $R_j$  (largeness can depend on $Q$), we define \[\label{defnofwj}w_j = \left|\left((\pi_j)_{*}(\frac{1}{R_j^{\frac{1}{N}}}\Z^N)_{R_j, 1}\right)^{\check{}}\right|^2 \left|\phi(\frac{\cdot}{R_j})\right|^2\cdot\left(1-\left|\psi(\frac{Q}{10R_j}\cdot)\right|^2\right).\]

    To define   $f_j$, we need a modulation (i.e. a translation on the Fourier side) and first define the amount of translation. For each $R_j$, define the finite set $D_j = \mathcal{N}_{\frac{1}{R_j}}(\Sigma)\cap\pi_j (\frac{1}{R_j^{\frac{1}{N}}}\Z^N \cap B_1)$.\footnote{If we work with a general $\tau_j$, we should define $D_j = \mathcal{N}_{\frac{1}{R_j}}(\Sigma)\cap\tau_j\pi_j (\frac{1}{R_j^{\frac{1}{N}}}\Z^N \cap B_1)$.} It is large by Lemma \ref{CZcountinglemma}: \[|D_j| \geq R_j^{\frac{n-1}{n-1+k}-\eps}.\] By pigeonholing, there exists $\xi_j \in D_j$ such that the set $E_j=\mathcal{N}_{\frac{1}{Q}} (\xi_j)\cap D_j$ satisfies \[\label{countingofEj}|E_j|=|\mathcal{N}_{\frac{1}{Q}} (\xi_j)\cap D_j| \gtrsim \frac{1}{Q^{n-1}} R_j^{\frac{n-1}{n-1+k}-\eps}.\] We will take this $\xi_j$ to be the amount of translation, and use crucially the fact that at every point in $\mathcal{N}_{\frac{2}{Q}} (\xi_j)\cap\Sigma$, the tangent hyperplane of $\Sigma$ is $O(\frac{1}{Q})$-close to a common hyperplane, denoted by $P_j$.

Now define \[f_j = \left((T_{\xi_j})_{*} (T_{-\xi_j} E_j)_{\frac{R_j}{Q}, \frac{1}{Q}}\right)^{\check{}}.\]

It suffices to explicitly check the failure of \eqref{MTversion} for $f_j$ and $w_j$. To this end, we bound both sides of \eqref{MTversion} explicitly and will mainly rely on computations on the Fourier side. 

We first work with the right-hand side. Let us begin by  understanding the supports of $f$ and $w$. For the support of $f_j$, we notice that $T_{\xi_j}$ corresponds to a modulation and plays no role. Thus $\text{supp}f_j = \text{supp}\left((T_{-\xi_j} E_j)_{\frac{R_j}{Q}, \frac{1}{Q}}\right)^{\check{}}$. We can use the projection-slicing theorem and see \[\text{supp} f_j \subset B_{\frac{2R_j}{Q}}\] because $\text{supp}\phi \subset B_2$. For the support of $w_j$, due to the physical cutoff factor in the end of expression \eqref{defnofwj}, we see \[\text{supp}w_j \subset B_{2R_j}\setminus B_{\frac{10R_j}{Q}}.\]

    Hence $\text{supp}f_j$ and $\text{supp}w_j$ are $\sim \frac{R_j}{Q}$-separated. This shows that when we consider $\mathcal{M}w_j$ on $\text{supp} f_j$, we only need to consider segments with $r\gtrsim \frac{R_j}{Q}$ in the definition \eqref{defnofmaximalfunction}.

    Based on this, we would like to give a strong bound of $\|\mathcal{M}w_j\|_{L^{\infty}(\text{supp}f_j)}$. We first upper bound the integration of $w_j$ on any line. We  can ignore the truncation $\left|\phi(\frac{\cdot}{R_j})\right|^2\cdot\left(1-\left|\psi(\frac{Q}{10R_j}\cdot)\right|^2\right)$, as this only makes the integration larger (up to a constant). Then, by duality, the maximal integration on a line is less than the maximal integration of the Fourier transform on a hyperplane, i.e. we have the bound \[\sup_{\Pi: \text{ hyperplane}}\int_{\Pi} \left((\pi_j)_{*}(\frac{1}{R_j^{\frac{1}{N}}}\Z^N)_{R_j, 1}\right) * \left((\pi_j)_{*}(\frac{1}{R_j^{\frac{1}{N}}}\Z^N)_{R_j, 1}\right).\]

    The integrand is a sum of weighted translations of $\left((\pi_j)_*(R_j^N \hat{\phi}(R_j \cdot))\right)*\left((\pi_j)_*(R_j^N \hat{\phi}(R_j\cdot))\right)$. Each of these is $L^1$-normalized and rapidly decays away from a scale of $R_j^{-1}$. The translation amounts are in a set of the shape $A_j+A_j$ where $A_j$ is the projection of $\frac{1}{R_j^{\frac{1}{N}}}\Z^N$ inside $B_{1}$ under $\pi_j$. This lies in $O(1)$ copies of $A_j$ with each point having multiplicity $O(|A_j|)$. By the conclusion of Lemma \ref{CZcountinglemma}, every hyperplane is $\frac{1}{R_j}$-close to $O(R_j^{\eps})$ many points in $A_j$. We can now conclude (with a standard summability control for all the tails on the Fourier side)
    \[\sup_{l: \text{ line}}\int_{l} w_j \mathrm{d}x \leq \sup_{\Pi: \text{ hyperplane}}\int_{\Pi} \left((\pi_j)_{*}(\frac{1}{R_j^{\frac{1}{N}}}\Z^N)_{R_j, 1}\right) * \left((\pi_j)_{*}(\frac{1}{R_j^{\frac{1}{N}}}\Z^N)_{R_j, 1}\right)\lesssim R_j^{1+\eps}|A_j|.\] Thus 
    \[\|\mathcal{M}w_j\|_{L^{\infty}(\text{supp}f_j)} \lesssim \frac{\sup_{l: \text{ line}}\int_{l} w_j \mathrm{d}x}{(R_j/Q)}\lesssim QR_j^{\eps}|A_j|\]
    and 
    \[\label{RHSestimate}\int_{\R^n} |f_j(x)|^2 \mathcal{M} w_j (x) \mathrm{d}x\lesssim QR_j^{\eps}|A_j|\|f_j\|_2^2 = QR_j^{\eps}|A_j|\|\hat{f}_j\|_2^2 \sim QR_j^{\eps} \cdot (\frac{R_j}{Q})^{2n}\cdot (\frac{R_j}{Q})^{-n} |A_j||E_j|\sim QR_j^{\eps} \cdot (\frac{R_j}{Q})^{n} |A_j||E_j|.\]

    Now we estimate $\int_{\R^n} |S_{\mathcal{N}_{\frac{1}{R_j}}(\Sigma)}f_j (x)|^2 w_j(x) \mathrm{d}x$. The shape of $w_j$ gives a main term and a minor term. We will show the minor term is much smaller than the main term, and at the same time obtain a good lower bound of the main term. These will be done by analyzing the Fourier side. Observe that by definition,
    \begin{eqnarray}\label{IminusII}
        &\int_{\R^n} |S_{\mathcal{N}_{\frac{1}{R_j}}(\Sigma)}f_j (x)|^2 w_j(x) \mathrm{d}x\nonumber\\ = &\|\left(\left((T_{\xi_j})_{*} (T_{-\xi_j} E_j)_{\frac{R_j}{Q}, \frac{1}{Q}}\right)\cdot 1_{\mathcal{N}_{\frac{1}{R_j}} (\Sigma)}\right)*\left((\pi_j)_{*}(\frac{1}{R_j^{\frac{1}{N}}}\Z^N)_{R_j, 1}\right)*\left(R_j^n\hat{\phi}(R_j\cdot)\right)\|_2^2 \nonumber\\
        - &\|\left(\left((T_{\xi_j})_{*} (T_{-\xi_j} E_j)_{\frac{R_j}{Q}, \frac{1}{Q}}\right)\cdot 1_{\mathcal{N}_{\frac{1}{R_j}} (\Sigma)}\right)*\left((\pi_j)_{*}(\frac{1}{R_j^{\frac{1}{N}}}\Z^N)_{R_j, 1}\right)*\left(R_j^n\hat{\phi}(R_j\cdot)\right)*\left((\frac{10R_j}{Q})^n\hat{\psi}(\frac{Q}{10R_j}\cdot)\right)\|_2^2\nonumber\\
        :=& I_1 - I_2.
    \end{eqnarray}

    The third convolution components in $I_1$ and $I_2$  are mollifications and we will discuss their effects after analyzing the first two terms. The second term is a sum of $L^1$-normalized bump functions at scale $\frac{1}{R_j}$ around points in $A_j$. For the first term, before the cutoff it is a sum of subfunctions around points in $E_j\subset A_j$ (and potentially more points in $A_j \setminus E_j$ in $B_{\frac{2}{Q}}$, with some suitable weight). Each subfunction is an $L^1$-normalized bump function at scale $\frac{Q}{R_j}$ but cutoff by approximately the $\frac{1}{R_j}$-neighborhood of a hyperplane parallel to $P_j$. Note that $A_j$ is an approximate group: $A_j+A_j$ can be covered by $O(1)$ copies of translations of $A_j$.  We deduce the convolution of the first two terms is lower-bounded by the sum of $\sim (\frac{R_j}{Q})^n |E_j|$ times characteristic functions of slabs of size $\sim\frac{1}{R_j}\times \frac{Q}{R_j}$ parallel to the hyperplane $P_j$ around $\gtrsim |A_j|$ many points in $A_j + A_j$, and is upper bounded by the sum of $\sim (\frac{R_j}{Q})^n |E_j|$ times of bump functions adapted to slabs of size $\sim\frac{1}{R_j}\times \frac{Q}{R_j}$ parallel to the hyperplane $P_j$ around each point in $A_j + A_j$. The convolution with $R_j^n\hat{\phi}(R_j\cdot)$ in $I_1$ does not change these upper and lower bounds, but the convolution with $(\frac{10R_j}{Q})^n\hat{\psi}(\frac{Q}{10R_j}\cdot)$ (bound this kernel by its absolute value) in $I_2$ mollifies the upper and lower bound by enlarging the essential support by $Q$ times and making the magnitude $\frac{1}{Q}$-times smaller. Note that by the separation condition in the end of Lemma \ref{CZcountinglemma}, as long as $\frac{Q}{R_j}<0.001R_j^{-\frac{1.5}{n}}$ (which we will always make our choice of $Q$ satisfy), all individual boxes above around different points in $A_j + A_j$ will not overlap. Hence we have
    \[I_1 \gtrsim \frac{1}{Q}(\frac{R_j}{Q})^n |E_j|^2 |A_j|\] and \[I_2 \lesssim \frac{1}{Q^2}(\frac{R_j}{Q})^n |E_j|^2 |A_j|.\] Thus we can now take $Q$ to be a large constant and have \[\label{LHSestimate}\int_{\R^n} |S_{\mathcal{N}_{\frac{1}{R_j}}(\Sigma)}f_j (x)|^2 w_j(x) \mathrm{d}x = I_1-I_2 \gtrsim \frac{1}{Q}(\frac{R_j}{Q})^n |E_j|^2 |A_j|.\]
    Combining \eqref{RHSestimate} and \eqref{LHSestimate}, we obtain
    \[\int_{\R^n} |S_{\mathcal{N}_{\frac{1}{R_j}}(\Sigma)}f_j (x)|^2 w_j(x) \mathrm{d}x \gtrsim Q^{-2}R_j^{-\eps}|E_j|\cdot \int_{\R^n} |f_j(x)|^2 \mathcal{M} w_j (x) \mathrm{d}x.\]

    Recall that $Q$ is a chosen constant and $|E_j|$ is lower bounded by \eqref{countingofEj} (and that $\eps$ can be arbitrarily small). We see we disprove \eqref{failedestimateforSK}. Examining this disproof, we see we  even construct examples where \eqref{correctededestimateforSK} holds as long as $R_j$ is sufficiently large.

    The most technical point of the proof is the lower bound of $I_1-I_2$ in \eqref{IminusII}, and we used a whole paragraph to motivate the analysis. Here is a helpful picture illustrating the difference between the supports of the integrand in $I_1$ (middle) and $I_2$ (right):

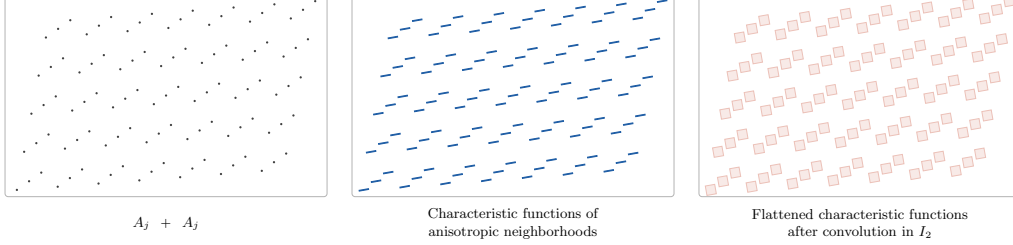
\scalebox{0.6}{
    \begin{tikzpicture}[font=\small]
  % Three generators of a proper planar GAP; each panel contains 105 points.
  % The enlarged neighborhoods are disjoint, as in the separation hypothesis.
  \def\ux{.90} \def\uy{.07}
  \def\vx{.16} \def\vy{.84}
  \pgfmathsetmacro{\wx}{.19*sqrt(2)}
  \pgfmathsetmacro{\wy}{.11*sqrt(3)}

  \foreach \panel/\shift in {0/0,1/7.65,2/15.30}{
    \begin{scope}[xshift=\shift cm]
      \draw[black!30,rounded corners=2pt,line width=.4pt]
        (-3.55,-2.23) rectangle (3.55,2.23);
      \foreach \i in {-3,...,3}
        \foreach \j in {-2,...,2}
          \foreach \k in {-1,...,1}{
            \pgfmathsetmacro{\px}{\i*\ux+\j*\vx+\k*\wx}
            \pgfmathsetmacro{\py}{\i*\uy+\j*\vy+\k*\wy}
            \begin{scope}[shift={(\px,\py)},rotate=10]
              \ifnum\panel=0
                \fill[black!75] (0,0) circle[radius=.025];
              \fi
              \ifnum\panel=1
                \filldraw[fill=packetblue!75,draw=packetblue,
                          line width=.18pt]
                  (-.105,-.009) rectangle (.105,.009);
              \fi
              \ifnum\panel=2
                % Pale fill represents reduced amplitude after convolution.
                \filldraw[fill=packetred!12,draw=packetred!35,
                          line width=.18pt]
                  (-.105,-.105) rectangle (.105,.105);
              \fi
            \end{scope}
          }
    \end{scope}
  }

  \node[align=center,text width=7.0cm] at (0,-2.82)
    {$A_j+A_j$};
  \node[align=center,text width=7.0cm] at (7.65,-2.82)
    {Characteristic functions of\\anisotropic neighborhoods};
  \node[align=center,text width=7.0cm] at (15.30,-2.82)
    {Flattened characteristic functions\\after convolution in $I_2$};
\end{tikzpicture}
}
\end{proof}

\section{Some other variants of Stein's conjecture}\label{variants}

In \cite{stein-conjecture-79}, Stein commented that for problems like Conjecture \ref{Steinconj},  it is not just their forms \emph{per se} that are important, but equally important are the spirit behind  and  possible research directions they convey. In this section we discuss some consequences of the method in our proof of Theorem \ref{Counterexthm} in some variants of Conjecture \ref{Steinconj}.

In \cite{cairo-counterexample-25}, Cairo disproved the MT conjecture. As opposed to Lemma \ref{CZcountinglemma}, for any $C^2$ hypersurface $\Sigma$ not in a hyperplane (in particular the sphere or truncated paraboloid) in $\R^n$ and any $R>1$, Cairo constructed $A, D \subset B_1$ with $|D| \gtrsim \log R$, $|A|\sim R^c$ ($0<c<1$ fixed) such that: 

(i) $A$ is an embedded Hamming cube of dimension $=|D|\sim \log R$.

(ii) $D$ consists of the generators of the Hamming cube and $|(x+A)\cap A|=\frac{|A|}{2}, \forall x \in D$.

(iii) $A$ intersects any $\frac{1}{R}$-slab at $O(1)$ points.

We comment that by the proof in \cite{cairo-counterexample-25}, one can further require:

(iv) Points in $A+A$ are pairwise $R^{-0.99}$-separated.\footnote{It is not hard to see that $0.99$ can be replaced by any positive number, but $0.99$ suffices.}

Let us see why we can have (iv). Cairo's choice of generators in $D$ all lie on a nondegenerate curve in $\R^n$, and lie in scales in a geometric progression of constant common ratio. The particular choices are flexible as long as we have one point per scale. Now if we only use scales between $R^{-\frac{1}{100n}}$ and $1$, we still have $\sim_n \log R$ many generators to use, and can easily use a crude greedy algorithm to make each point pair (there are less than $R^{0.1}$ such pairs) in $A+A$ to be  $R^{-0.99}$-separated.

If we construct the $E_j$ and $A_j$ based on the above $D$ and $A$ satisfying (i)-(iv), by the same proof of Theorem \ref{Counterexthm}, we obtain a counterexample to Stein's original Conjecture \ref{Steinconj} at the endpoint $\delta=0$ case:

\begin{theorem}\label{Counterexthmendpt}
    For any given dimension $n \geq 2$  and every $R>1$, there exist $f$ and $w$ supported in $B_R$ such that 
    \[\label{correctededestimateforSKlog}\int_{\R^n} |Sf (x)|^2 w(x) \mathrm{d}x \geq \log R\cdot  \int_{\R^n} |f(x)|^2 \mathcal{M} w (x) \mathrm{d}x.\]

    Moreover, there is a similar counterexample $f_K$ if we replace $S$ by $S_K$, where $K$ is any convex body with $C^2$ boundary.
\end{theorem}

\begin{proof}[Proof sketch]
    We recall the procedure to construct $f$ and $w$ for Theorem \ref{Counterexthm}, and use the same kind of construction here. $\hat{f}$ will be like the sum of  bump functions around points in $E_j$, and $w = |\check{b}|^2$ where $b$ will behave like the sum of  bump functions around $A_j$. Now property (i) and (ii) guarantee that $\hat{f}*b$ has (logarithmic) many collisions at a positive proportion of points in $A_j$, (iii) is enough to guarantee that the X-ray transform of $w$ is small in $L^{\infty}$ by a short computation via the projection-slicing theorem, and we need a final technical step to separate the supports of $f$ and $w$ and introduce slightly anisotropic, slab-like  neighborhoods. We introduce a parameter $Q$ and use pigeonholing to make sure slabs in the same direction pile up in the analogue of integral $I_1$ in \eqref{IminusII}, and in the analogue of integral $I_2$, we use separation condition (iv) to ensure a good upper bound.
\end{proof}

Recently, there is a follow-up work \cite{bennett2026rectangles} that, among other things, improves Cairo's lower bound for the paraboloid to $(\log R)^{n-1}$ in dimension $n$. This may be used to construct lower weighted bounds similar to Theorem \ref{Counterexthmendpt} for multipliers associated with a convex body whose boundary contains a paraboloid patch. See also related recent works \cite{fenves2026cusp, fenves2026maximal} to \cite{cairo2025power}.

\begin{remark}
    The author used GPT-6 Astra to help proofread the present article. GPT-6 Astra suggested a more elementary way to generate a logarithmic lower bound for all solids with a $C^{\alpha} (\alpha>0)$ boundary patch. The argument, of de Leeuw-type \cite{deleeuw-multipliers-65}, is included in Appendix \ref{app:direct-logarithmic-bound}. The one-dimensional case is related to work of Wilson \cite{wilson1989weighted} and P\'{e}rez \cite{perez1994weighted}.
\end{remark}

Stein's original problem (5(b) and (c) in \cite{stein-conjecture-79}) also asked for the non-endpoint case (where one considers the Bochner-Riesz multiplier \[S^{\delta}: f \mapsto \left((1-|\xi|^2)_+^{\delta} \hat{f}\right)^{\check{}}\] when $\delta > 0$). The proof of Theorem \ref{Counterexthm} shows this fails for some convex body: The difference between the analogously defined $S_{K}^{\delta}$ and a rescaling by $1+O(\frac{1}{R})$ gives approximately $R^{-\delta} S_{\mathcal{N}_{\frac{1}{R}}(\Sigma)}$ We thus see that the variants of these are false for some convex bodies with $C^k$-boundary for arbitrarily high $k$ whenever $0< \delta< \frac{n-1}{2(n-1+k)}$. However, the original problem for the Bochner-Riesz multiplier remains open.

In dimension $1$, Conjecture \ref{Steinconj} is true with the maximal function on the right-hand side being replaced by $(\mathcal{M} w^s)^{\frac{1}{s}}$ for any $s>1$ (see also \cite{bcsv-stein-conjecture-06}). Because the weights $w_j$ in our counterexample in the last section have a crude $L^{\infty}$ bound, we easily see that in higher dimensions they remain counterexamples for the version of Conjecture \ref{Steinconj} with $(\mathcal{M} w^s)^{\frac{1}{s}}$ for $s>1$ very close to $1$.

\appendix
\section{An elementary logarithmic lower bound for domains with H\"older boundary}
\label{app:direct-logarithmic-bound}
\numberwithin{equation}{section}

\definecolor{directlogblue}{HTML}{286A9A}
\definecolor{directlogpurple}{HTML}{8C4678}
\definecolor{directlogink}{HTML}{273746}

This Appendix is generated by GPT-6 Astra, with correctness checked and minor revisions by the author. The argument is of de Leeuw-type \cite{deleeuw-multipliers-65}. As a warm-up,  we first include a one-dimensional argument that is also related to the work of P\'{e}rez and Wilson \cite{perez1994weighted, wilson1989weighted}.

For a compact set \(K\subset\mathbb R^n\), define
\begin{equation}
\label{eq:direct-log-fourier-definition}
 \widehat f(\xi)=\int_{\mathbb R^n}f(x)e^{-2\pi\mathrm i x\cdot\xi}\,d x,
 \qquad S_Kf(x)=\int_K\widehat f(\xi)e^{2\pi\mathrm i x\cdot\xi}\,d\xi.
\end{equation}
For \(w\ge0\), let \(\mathcal{M}w\) be the largest average along a centered line segment:
\begin{equation}
\label{eq:direct-log-maximal-definition}
 Mw(x)=\sup_{\substack{|\omega|=1\\r>0}}\frac1{2r}\int_{-r}^{r}w(x-s\omega)\,d s.
\end{equation}
\begin{theorem}
Let \(n\ge2\). Suppose that near some boundary point, in suitable orthogonal
coordinates, \(K\) is the region below the graph of a \(C^\alpha\) function,
for some \(\alpha>0\). Then there is \(R_0\) such that,
for every \(R\ge R_0\), there exist nonzero \(f_R,w_R\in C_c^\infty(B_R)\),
with \(0\le w_R\le1\), satisfying
\begin{equation}
 \int |S_Kf_R|^2w_R\gtrsim \log R\int |f_R|^2Mw_R.
 %\qquad B_R=\{x:|x|<R\}.
 \label{eq:direct-log-1}
\end{equation}
In particular, this holds for every compact solid with \(C^\alpha\) boundary.
%For \(n=1\), it holds when \(K\) is any bounded interval of positive length.
\end{theorem}
Only the stated graph patch is required; the rest of \(K\) may have arbitrary
geometry. %Constants may depend on \(K,\alpha\), and fixed cutoffs, but not on \(R\). 
By an analogous proof, we also see that if we replace the $C^{\alpha}$ assumption by mere continuity, the quotient of both sides of \eqref{eq:direct-log-1} can be shown to be unbounded.

\subsection*{Proof in one dimension}
\begin{proof}[Proof of \eqref{eq:direct-log-1} in one dimension]
After modulation, take \(K=[-a,0]\), where \(a>0\). Choose smooth
\(0\le h\le1\), supported in \([1,R/2]\), equal to \(1\) on \([2,R/3]\),
with monotone transitions, so \(\int|h'|=2\). Fix a nonzero smooth
\(0\le \psi\le1\) supported in \((-1,0)\).
For \(-1\le t\le0\), integration over the frequency interval gives
\begin{equation}
 S_Kh(t)=\frac1{2\pi\mathrm i}\int_1^{R/2}h(u)
       \frac{1-e^{-2\pi\mathrm i a(t-u)}}{t-u}\,d u
       =-\frac{\log R}{2\pi\mathrm i}+O_a(1).
 \label{eq:direct-log-2}
\end{equation}
Indeed, the integral of \(h(u)/(t-u)\) is \(-\log R+O(1)\). %: replacing \(h\) by \(\mathbf1_{[1,R/2]}\) costs at most \(\log3\).
The remaining oscillating integral, by integration by parts, has absolute value at most
\begin{equation}
\label{eq:direct-log-oscillatory-bound}
 \frac1{2\pi a}\int_1^{R/2}
 \left(\frac{|h'(u)|}{u-t}+\frac{h(u)}{(u-t)^2}\right)\,d u
 \le\frac3{2\pi a}.
\end{equation}
At \(u\ge1\), %a centered interval reaching \([-1,0]\) has half-length at least \(u\) and contains at most one unit of weight. Thus 
we see by definition that \(\mathcal{M} \psi(u)\le1/(2u)\), and
\begin{equation}
\label{eq:direct-log-one-dimensional-integrals}
 \int|S_Kh|^2 \psi\gtrsim_a(\log R)^2,
 \qquad \int|h|^2\mathcal{M} \psi\le\frac12\int_1^{R/2}\frac{\,d u}{u}
 \lesssim\log R.
\end{equation}
Taking \(f_R=h\) and \(w_R=\psi\) proves \eqref{eq:direct-log-1} in dimension $1$.
\end{proof}

\subsection*{The higher-dimensional construction}
We recall that in the above construction and computation \eqref{eq:direct-log-2}, each dyadic interval in $\text{supp} h$ contributes a constant to the
nonoscillating integral. These contributions were added up and then squared. On the other hand, the maximal average contributes only the integral of $\frac1u$. We construct a higher-dimensional analogue by taking a product example of this construction and balls around the origin.

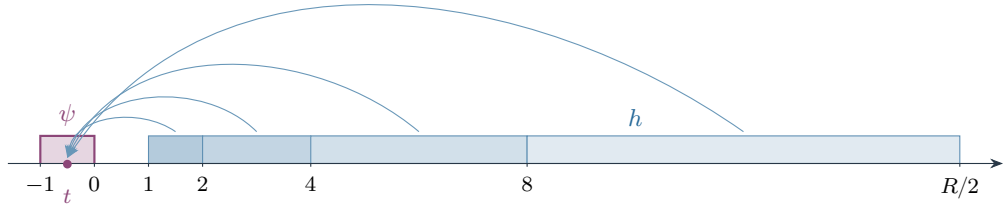
\begin{center}
\resizebox{0.88\linewidth}{!}{%
\begin{tikzpicture}[x=0.70cm,y=0.42cm,>=Stealth,font=\small]
 \fill[directlogpurple!22] (-1,0) rectangle (0,0.85);
 \draw[directlogpurple,thick] (-1,0)--(-1,0.85)--(0,0.85)--(0,0);
 \foreach \a/\b/\opacity in {1/2/35,2/4/28,4/8/21,8/16/14}{
  \fill[directlogblue!\opacity] (\a,0) rectangle (\b,0.85);
  \draw[directlogblue!75] (\a,0) rectangle (\b,0.85);
 }
 \draw[->,directlogink] (-1.6,0)--(16.8,0);
 \foreach \x/\lab in {-1/{-1},0/0,1/1,2/2,4/4,8/8,16/{R/2}}{
   \draw[directlogink] (\x,0)--(\x,-0.10);
   \node[below,font=\footnotesize] at (\x,-0.12) {$\lab$};
 }
 \node[directlogpurple,above] at (-0.5,0.86) {$\psi$};
 \node[directlogblue,above] at (10,0.86) {$h$};
 \fill[directlogpurple] (-0.5,-0.02) circle (1.6pt);
 \node[below,directlogpurple,font=\footnotesize] at (-0.5,-0.52) {$t$};
 \draw[->,directlogblue!70] (1.5,0.98) to[out=145,in=80] (-0.5,0.17);
 \draw[->,directlogblue!70] (3,0.98) to[out=135,in=72] (-0.5,0.17);
 \draw[->,directlogblue!70] (6,0.98) to[out=140,in=64] (-0.5,0.17);
 \draw[->,directlogblue!70] (12,0.98) to[out=145,in=56] (-0.5,0.17);
\end{tikzpicture}
}
\end{center}
\begin{center}
\small The input lies to the right of the weight. The subdivision shows the contributions from all dyadic intervals.
\end{center}

 %The point to check is that moving the endpoint of the frequency interval changes the output by only a bounded amount. We will also bound the contribution from the rest of \(K\).

It suffices to prove the theorem for \(0<\alpha\le1\). In the proof, we suppress dependences on $\alpha$. %; when the given regularity exponent is greater than \(1\), use the argument with exponent \(1\).
Translate the chosen frequency point to \(0\) and use its graph coordinates.
This amounts to a modulation and rotation in physical space, preserving \eqref{eq:direct-log-1}
and supports in \(B_R\). Write frequency coordinates as \((\sigma,\eta)\)
and physical coordinates  \((u,z)\).
For fixed \(a,\rho>0\), the local graph condition is
\begin{equation}
 K\cap\{ |\sigma|<a,\ |\eta|<\rho\}
 =\{(\sigma,\eta):|\eta|<\rho,\ -a<\sigma\le\beta(\eta)\},
 \label{eq:direct-log-3}
\end{equation}
where by the $C^{\alpha}$ assumption,
\begin{equation}
 \beta(0)=0,\qquad |\beta(\eta)|\le C_K|\eta|^\alpha,
 \qquad |\beta(\eta)|<a/2\quad (|\eta|<\rho).
 \label{eq:direct-log-4}
\end{equation}
Here \(a\) is the fixed vertical half-width of the cylinder, and \(\rho\)
is its fixed transverse radius. %For \(0<\alpha\le1\), a \(C^\alpha\) graph satisfies \(|\beta(\eta)-\beta(\eta')|\le C|\eta-\eta'|^\alpha\), which gives \eqref{eq:direct-log-4}.

For the remainder of the proof, choose \(h\) as above with \(R\) replaced
by \(R^\alpha\):
\begin{equation}
\label{eq:direct-log-longitudinal-cutoff}
 0\le h\le1,\quad \operatorname{supp}h\subset[1,R^\alpha/2],\quad
 h=1\text{ on }[2,R^\alpha/3],\quad \int|h'|=2.
\end{equation}
Fix smooth \(0\le g\le1\) on \(\mathbb R^{n-1}\), equal to \(1\) on \(B_{1/2}\)
and supported in \(B_1\). Using the same \(\psi\) as above, put
\begin{equation}
 g_R(z)=g(4z/R),\qquad
 f_R(u,z)=h(u)g_R(z),\qquad w_R(t,z)=\psi(t)g_{\frac{R}{2}} (z).
 \label{eq:direct-log-5}
\end{equation}
Both supports lie in \(B_R\) for large \(R\). %On the support of \(w_R\), \(g_R(z)=1\), and \(\int w_R\sim R^{n-1}\).

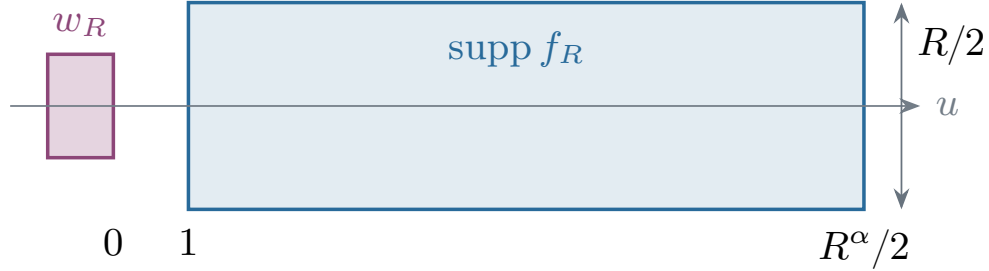
\begin{center}
\resizebox{0.88\linewidth}{!}{%
\begin{tikzpicture}[x=0.75cm,y=0.57cm,>=Stealth,font=\small]
 \fill[directlogblue!13] (0.8,-1.45) rectangle (8,1.45);
 \draw[directlogblue,thick] (0.8,-1.45) rectangle (8,1.45);
 \fill[directlogpurple!23] (-0.7,-0.725) rectangle (0,0.725);
 \draw[directlogpurple,thick] (-0.7,-0.725) rectangle (0,0.725);
 \draw[->,directlogink!65] (-1.1,0)--(8.6,0) node[right] {$u$};
 \node[directlogblue] at (4.3,0.8) {$\operatorname{supp} f_R$};
 \node[directlogpurple,above] at (-0.35,0.77) {$w_R$};
 \node[below] at (0,-1.52) {$0$};
 \node[below] at (0.8,-1.52) {$1$};
 \node[below] at (8,-1.52) {$R^\alpha/2$};
 \draw[<->,directlogink!70] (8.4,-1.45)--(8.4,1.45);
 \node[right] at (8.4,0.85) {$R/2$};
\end{tikzpicture}
}
\end{center}
\begin{center}
\small A two-dimensional slice of the supports. The weight has half the
transverse radius of the input. The drawing is not to scale.
\end{center}

\subsection*{An exact formula using the graph patch}
Write \(K_\eta=\{\sigma:(\sigma,\eta)\in K\}\). These sections need not be intervals but contain intervals from the right when $\eta$ is close to $0$.

\begin{lemma}
For the functions in \eqref{eq:direct-log-5}, \(-1\le t\le0\), and \(z\in\mathbb R^{n-1}\),
\begin{equation}
 S_Kf_R(t,z)=\int_{\mathbb R^{n-1}}e^{2\pi\mathrm i z\cdot\eta}\widehat g_R(\eta)
       \int_{K_\eta}e^{2\pi\mathrm i t\sigma}\widehat h(\sigma)\,d\sigma\,d\eta
 \label{eq:direct-log-6}
\end{equation}
and, more explicitly,
\begin{equation}
 \boxed{\begin{aligned}
 S_Kf_R(t,z)=\frac1{2\pi\mathrm i}\int_{|\eta|<\rho}\int_1^{R^\alpha/2}
 &e^{2\pi\mathrm i z\cdot\eta}\widehat g_R(\eta)h(u)\\[-2pt]
 &\times\frac{e^{2\pi\mathrm i\beta(\eta)(t-u)}-e^{-2\pi\mathrm i a(t-u)}}{t-u}
 \,d u\,d\eta+E(t,z),
 \end{aligned}}
 \label{eq:direct-log-7}
\end{equation}
where the remainder is
\begin{equation}
\label{eq:direct-log-remainder}
\begin{aligned}
 E(t,z)={}&\int_{|\eta|<\rho}e^{2\pi\mathrm i z\cdot\eta}\widehat g_R(\eta)
       \int_{K_\eta\cap\{|\sigma|\ge a\}}e^{2\pi\mathrm i t\sigma}\widehat h(\sigma)\,d\sigma\,d\eta\\
 &+\int_{|\eta|\ge\rho}e^{2\pi\mathrm i z\cdot\eta}\widehat g_R(\eta)
       \int_{K_\eta}e^{2\pi\mathrm i t\sigma}\widehat h(\sigma)\,d\sigma\,d\eta.
\end{aligned}
\end{equation}
All these integrals converge absolutely to $O_{K, g} (1)$.
\end{lemma}
\begin{proof}
Equation \eqref{eq:direct-log-6} is Fourier inversion applied to
\(\widehat f_R(\sigma,\eta)=\widehat h(\sigma)\widehat g_R(\eta)\).
On \(|\eta|<\rho\), \eqref{eq:direct-log-3} identifies the part with \(|\sigma|<a\) as
\([-a,\beta(\eta)]\), up to endpoints. Expanding \(\widehat h\) gives
\begin{equation}
\label{eq:direct-log-interval-formula}
 \int_{-a}^{\beta(\eta)}e^{2\pi\mathrm i t\sigma}\widehat h(\sigma)\,d\sigma
 =\frac1{2\pi\mathrm i}\int_1^{R^\alpha/2}h(u)
 \frac{e^{2\pi\mathrm i\beta(\eta)(t-u)}-e^{-2\pi\mathrm i a(t-u)}}{t-u}\,d u.
\end{equation}
Splitting \eqref{eq:direct-log-6} by these frequency regions proves \eqref{eq:direct-log-7} and the displayed formula for \(E\).

To bound \(E\), compact support and one integration by parts give
\begin{equation}
\label{eq:direct-log-fourier-decay}
 |\widehat h(\sigma)|\le\min\{\|h\|_1,\|h'\|_1/(2\pi|\sigma|)\}
 \le C\min\{R^\alpha,|\sigma|^{-1}\}.
\end{equation}
Since \(K\) is bounded, uniformly in \(\eta\),
\begin{equation}
 \int_{K_\eta\cap\{|\sigma|\ge a\}}|\widehat h(\sigma)|\,d\sigma\le C_K,
 \qquad
 \int_{K_\eta}|\widehat h(\sigma)|\,d\sigma\le C_K\log R.
 \label{eq:direct-log-8}
\end{equation}
Also, \(\widehat g_R(\eta)=(R/4)^{n-1}\widehat g(R\eta/4)\), so
\begin{equation}
 \int|\widehat g_R|=\int|\widehat g|,
 \qquad \int_{|\eta|\ge\rho}|\widehat g_R(\eta)|\,d\eta\le C_{K,g}R^{-2}.
 \label{eq:direct-log-9}
\end{equation}
The second bound follows, after rescaling, from
\(|\widehat g(\eta)|\le C_g(1+|\eta|)^{-n-3}\).
Thus the first integral defining \(E\) is bounded by \(C_{K,g}\), and the
second by \(C_{K,g}R^{-2}\log R\). This proves the claim; the same bounds
justify the integrations.
\end{proof}

\subsection*{Proof of the theorem in higher dimensions}
\begin{proof}[Proof of \eqref{eq:direct-log-1} in higher dimensions]
In \eqref{eq:direct-log-7}, split the numerator as
\(1+(e^{2\pi\mathrm i\beta(\eta)(t-u)}-1)-e^{-2\pi\mathrm i a(t-u)}\). This gives
\begin{equation}
 S_Kf_R=T_0+T_\beta+T_{-a}+E.
 \label{eq:direct-log-10}
\end{equation}
We define and estimate these three remaining terms below. All bounds are
uniform for \(-1\le t\le0\) and every \(z\).

\textit{The main term.}
\begin{equation}
\label{eq:direct-log-main-term}
 T_0(t,z):=\frac1{2\pi\mathrm i}
 \left(\int_{|\eta|<\rho}e^{2\pi\mathrm i z\cdot\eta}\widehat g_R(\eta)\,d\eta\right)
 \left(\int_1^{R^\alpha/2}\frac{h(u)}{t-u}\,d u\right)
 =-\frac{\alpha\log R}{2\pi\mathrm i}g_R(z)+O(1).
\end{equation}
Indeed, the second factor in parentheses is \(-\alpha\log R+O(1)\), by the one-dimensional proof.
By \eqref{eq:direct-log-9}, the first factor is \(g_R(z)+O(R^{-2})\).

\textit{The moving endpoint.}
\begin{equation}
\label{eq:direct-log-moving-endpoint}
 T_\beta(t,z):=\frac1{2\pi\mathrm i}
 \int_{|\eta|<\rho}\int_1^{R^\alpha/2}e^{2\pi\mathrm i z\cdot\eta}\widehat g_R(\eta)h(u)
 \frac{e^{2\pi\mathrm i\beta(\eta)(t-u)}-1}{t-u}\,d u\,d\eta.
\end{equation}
The inequality \(|e^{\mathrm i s}-1|\le|s|\) cancels the denominator \(u-t\). Hence
\begin{equation}
\label{eq:direct-log-moving-endpoint-bound}
 \begin{aligned}
 |T_\beta(t,z)|
 &\le C_KR^\alpha\int |\eta|^\alpha|\widehat g_R(\eta)|\,d\eta\\
 &=C_K4^\alpha\int |\eta|^\alpha|\widehat g(\eta)|\,d\eta\le C_{K,g}.
 \end{aligned}
\end{equation}
The rescaling cancels the factor \(R^\alpha\). This is the only place where
the H\"older exponent determines the length of the input. We remark that if we only assume $\beta$ is continuous and would like to derive unboundedness of the quotient of both sides of \eqref{eq:direct-log-1}, then by a short modification of the present argument, we can take the length of support of $h$ to be $L = \frac{1}{R^{-1}+\sup_{|\eta|\leq R^{-\frac12}} |\beta (\eta)|}$ (instead of $R^{\alpha}$), and see that this term has a fixed upper bound. This $L$ goes to $\infty$ when $R \to \infty$, enough to ensure unboundedness of the main term.

\textit{The fixed endpoint.}
\begin{equation}
\label{eq:direct-log-fixed-endpoint}
 T_{-a}(t,z):=-\frac1{2\pi\mathrm i}
 \int_{|\eta|<\rho}\int_1^{R^\alpha/2}e^{2\pi\mathrm i z\cdot\eta}\widehat g_R(\eta)h(u)
 \frac{e^{-2\pi\mathrm i a(t-u)}}{t-u}\,d u\,d\eta.
\end{equation}
The same integration by parts as in the one-dimensional proof bounds its
inner integral by \(3/(2\pi a)\). Thus \eqref{eq:direct-log-9} gives \(|T_{-a}|\le C_{K,g}\).
Combining these estimates with the lemma proves
\begin{equation}
 \boxed{\displaystyle
 S_Kf_R(t,z)=-\frac{\alpha\log R}{2\pi\mathrm i}g_R(z)+O_{K,g}(1),
 \qquad -1\le t\le0.}
 \label{eq:direct-log-11}
\end{equation}

It remains to compare the weighted integrals. The weight is supported in
\(-1\le t\le0\). A segment centered at \((u,z)\), \(u\ge1\), reaching this slab
has half-length at least \(u/|\omega_1|\), and its intersection with the slab
has length at most \(\frac{1}{|\omega_1|}\). Therefore
\(\mathcal{M}w_R(u,z)\le \frac{1}{2u}\); directions with \(\omega_1=0\) miss the slab.
Since \(g_R=1\) on the support of \(w_R\), \eqref{eq:direct-log-11} and \eqref{eq:direct-log-5} yield
\begin{equation}
\label{eq:direct-log-weighted-integrals}
 \int|S_Kf_R|^2w_R\gtrsim R^{n-1}(\alpha\log R)^2,
 \qquad
 \int|f_R|^2Mw_R\lesssim R^{n-1}\alpha\log R.
\end{equation}
For the fixed positive exponent \(\alpha\), these inequalities imply \eqref{eq:direct-log-1}.
\end{proof}

\printbibliography
\end{document}